\documentclass[11pt]{amsart}
\usepackage{a4wide,color}

\usepackage{amssymb}
\usepackage[leqno]{amsmath}
\usepackage{amsfonts}
\usepackage{amsopn}
\usepackage{amstext}
\usepackage{amsthm}
\usepackage{eucal}

\usepackage{tikz}
\usetikzlibrary{cd}
\usepackage[all]{xy}
\newdir{ >}{{}*!/-9pt/@{>}}

\newcommand{\conv}{\operatorname{conv}}

\newcommand{\Gurarii}{Gurari\u\i{ }}

\newcommand{\G}{\mathbb G}
\newcommand{\IG}{\mathbb G}

\newcommand{\IR}{\mathbb R}
\newcommand{\IQ}{\mathbb Q}

\newcommand{\Id}{\mathrm{Id}}
\newcommand{\Iso}{\mathrm{Iso}}

\newcommand{\U}{\mathcal U}

\newcommand{\F}{\mathcal F}
\newcommand{\w}{\omega}
\newcommand{\eps}{\varepsilon}
\newcommand{\IN}{\mathbb N}

\newcommand{\e}{\varepsilon}
\newcommand{\Ra}{\Rightarrow}

\newtheorem{theorem}{Theorem}
\newtheorem{corollary}{Corollary}
\newtheorem{lemma}{Lemma}
\newtheorem{claim}{Claim}

\newtheorem{proposition}{Proposition}

\theoremstyle{definition}
\newtheorem{remark}{Remark}

\title{Gurari\u\i\ operators are generic}
\author[ T.~Banakh, J.~Garbuli\'nska-W\c egrzyn]{ Taras Banakh$^{1,2}$ and  Joanna Garbuli\'nska-W\c egrzyn$^{2}$}
\address{$^{1}$Ivan Franko National University of Lviv, Ukraine}
\email{t.o.banakh@gmail.com}
\address{$^{2}$Jan Kochanowski University, Kielce, Poland}
\email{jgarbulinska@ujk.edu.pl}
\begin{document}

\begin{abstract} 
Answering a question of Garbuli\'nska-W\c egrzyn and Kubi\'s, we prove that \Gurarii operators form a dense $G_\delta$-set in the space $\mathcal B(\IG)$ of all nonexpansive operators on the \Gurarii space $\G$, endowed with the strong operator topology. This implies that  universal operators on  $\G$ form a residual set in $\mathcal B(\G)$.
\end{abstract}
\maketitle

\section{Introduction}
All {\em Banach spaces} considered in this paper are real and separable, all {\em operators} are linear bounded operators between separable Banach spaces. An {\em isometric embedding} (resp. {\em $\eps$-isometric embedding} for some $\eps>0$) of Banach spaces is any injective operator $T:X\to Y$ such that $\|T(x)\|=\|x\|$ (resp. $(1-\eps)\|x\|<\|T(x)\|< (1+\eps)\|x\|$) for any nonzero $x\in X$. An operator $T:X\to Y$ is called {\em nonexpansive} if $\|T\|\le 1$.

A Banach space $X$ is called {\em \Gurarii} if for any finite-dimensional Banach spaces $A\subseteq B$ and any $\eps>0$, any isometric embedding $i:B\to X$ extends to an $\eps$-isometric embedding $\bar f:A\to X$. The first example of a \Gurarii space $\G$ was constructed by \Gurarii in \cite{gurarii}. In \cite{lusky} Lusky proved that any two \Gurarii spaces are isometrically isomorphic, so up to an isometry there exists a unique \Gurarii space $\G$. A simple  proof of the uniqueness of the \Gurarii space was found by Kubi\'s and Solecki \cite{KS}. In \cite{Kubis}, Kubi\'s suggested an elementary game-theoretic contruction of the \Gurarii space. It is known \cite{gevorkian} that the \Gurarii space is {\em universal} in the sense that it contains an isometric copy of any separable Banach space. 

An operator $U:V\to W$ between Banach spaces is defined to be \emph{universal} if for every  operator $T:X\to Y$ with  $\|T\| \leq \| U\|$, there exist linear isometric embeddings $ i: X\to V$, $j: Y\to W$
such that the diagram
$$\begin{tikzcd}
	V \ar[rr, "U"] & & W \\
	X \ar[rr, "T"']\ar[u, hook, "i"] & & Y \ar[u, hook, "j"']
\end{tikzcd}$$
is commutative, that is, $U \circ  i = j \circ T$.

In \cite{GK}, \cite{GK2} Garbuli\'nska--W\c egrzyn and Kubi\'s constructed a universal operator $\Omega:\G\to\G$ using the technique of Fraisee limits. More precisely, they defined the notion of a \Gurarii operator (which is an operator counterpart of the notion of a \Gurarii space), constructed a \Gurarii operator (as the Fraisse limit in a suitable category) and proved that every \Gurarii operator is universal.

An operator $G:X\to Y$ between Banach spaces is called {\em \Gurarii\!} if $G$ is nonexpansive and for any $\eps>0$, any nonexpansive operator $T:A\to B$ between finite-dimensional Banach spaces, any Banach subspaces $A_0\subseteq A$, $B_0\subseteq B$ with $T[A_0]\subseteq B_0$, and any isometric embeddings   $i_0:A_0\subseteq X$, $j_0:B_0\subseteq Y$ with $G\circ i_0=j_0\circ T{\restriction}_{A_0}$, there exist $\e$-isometric embeddings $i:A\to X$ and $j:B\to Y$ such that $i{\restriction}_{A_0}=i_0$, $j{\restriction}_{B_0}=j_0$ and $G\circ i=j\circ T$.

An example of a \Gurarii operator $\Omega:\G\to \G$ was constructed in \cite{GK}. By \cite[Theorem 3.5]{GK}, any \Gurarii operator $G:X\to Y$ is isometric to the \Gurarii operator $\Omega:\G\to\G$ in the sense that there exist bijective isometries $i:X\to\G$ and $j:Y\to\G$ such that $\Omega\circ i=j\circ G$. Therefore, a \Gurarii operator is unique up to an isometry (like the \Gurarii space). By \cite[Theorem 3.3]{GK}, every \Gurarii operator is universal. Therefore, the set $\mathcal G(\G)$ of \Gurarii operators from a \Gurarii space $\G$ to itself is a subset of the set $\U(\IG)$  of nonexpansive universal operators from $\G$ to $\G$. In its turn, $\U(\IG)$ is a subset of the space $\mathcal B(\IG)$ of all nonexpansive operators from $\G$ to $\G$. The space $\mathcal B(\IG)$ is endowed with the strong operator topology. Let us recall that the {\em strong operator topology} on $\mathcal B(\IG)$ is the smallest topology such that for every $x\in \G$ the map $\delta_x:\mathcal B(\G)\to\IG$, $\delta_x:T\mapsto T(x)$, is continuous.

In \cite[Question 1]{GK2} Garbuli\'nska-W\c egrzyn and Kubi\'s asked whether the set $\mathcal G(\G)$ of \Gurarii operators is residual in the space $\mathcal B(\IG)$. Let us recall that a subset $R$ of a topological space $X$ is {\em residual} if it contains a dense $G_\delta$-subset of $X$.

In this paper we answer Question 1 of \cite{GK2} affirmatively proving the following main result.

\begin{theorem}\label{t:main} The set $\mathcal G(\G)$ of \Gurarii operators is dense $G_\delta$ in the space $\mathcal B(\IG)$ of all nonexpansive operators on $\IG$. 
\end{theorem}

Since $\mathcal G(\IG)\subseteq\U(\IG)$, Theorem~\ref{t:main} implies the following corollary.

\begin{corollary} The set $\U(\IG)$ of universal nonexpansive operators on $\IG$ is residual in the space $\mathcal B(\IG)$ of all nonexpansive operators on $\IG$.
\end{corollary}

\begin{remark} Observe that the space $\mathcal B(\G)$ is a topological semigroup with respect to the operation of composition of operators. The group $\Iso(\G)$ of linear isometries of $\G$ is Polish \cite[9.3]{Kechris} and hence is a $G_\delta$-subgroup of the topological semigroup $\mathcal B(\G)$. By \cite[3.5]{GK}, any \Gurarii operators are isometric, which implies the set $\mathcal G(\G)$ coincides with the orbit $\{I\circ \Omega\circ J:I,J\in\Iso(\G)\}$ of any \Gurarii operator $\Omega:\G\to\G$ under the two-sided action of the Polish group $\Iso(\G)\times\Iso(\G)$ on $\mathcal B(\G)$, and this orbit is a unique dense $G_\delta$-orbit of this action. Theorem~\ref{t:main} also shows that \Gurarii operators are ``generic'' elements of the space $\mathcal B(\IG)$ in the Baire category sense.
\end{remark}

\section{Preliminaries}

By $\IR_+$ we denote the half-line $[0,\infty)$. Each natural number $n$ is identified with the set $\{0,\dots,n-1\}$. For a function $f:X\to Y$ between sets and a subset $A\subseteq X$ we denote by $f[A]$ the image $\{f(x):x\in A\}$ of the set $A$ under the map $f$.

For a Banach space $X$, point $x\in X$ and positive real number $r$, let
$$B_X(x,r)=\{y\in X:\|x-y\|<r\}$$be the open ball of radius $r$ with center $x$ in the Banach space $X$. If $x=0$, then we shall write $B_X(r)$ instead of $B_X(0,r)$.

For an operator $T:X\to Y$ between Banach spaces, let
$$
\begin{aligned}
&\|T\|=\inf\{r\in\IR_+:T[B_X(1)]\subseteq B_Y(r)\}\quad\mbox{and}\\
&\|T^{\leftarrow}\|=\inf\big(\{r\in\IR_+:T[X]\cap B_Y(1)\subseteq T[B_X(r)]\}\cup\{\infty\}\big).
\end{aligned}
$$
The (finite or infinite) number $\|T^\leftarrow\|$ is called the {\em openness norm} of $T$. Observe that an operator $T:X\to Y$ between Banach spaces has finite openness norm $\|T^\leftarrow\|$ if and only if the operator $T:X\to T[X]$ is open if and only if $T[X]$ is a closed subspace of $Y$. In particular, for any operator $T:X\to Y$ from a finite-dimensional Banach space $X$, the openness norm $\|T^{\leftarrow}\|$ is finite. For a bijective operator $T:X\to Y$ the openness norm $\|T^{\leftarrow}\|$ is equal to the norm $\|T^{-1}\|$ of the inverse operator $T^{-1}:Y\to X$.

For a non-negative integer $n$ and a real number $p\in[1,\infty)$, let $\ell^n_p$ be the Banach space $\IR^n$ endowed with the $\ell_p$-norm
$$\|x\|=\Big(\sum_{i\in n}|x_i|^p\Big)^{\frac1p}\quad\mbox{for any $x=(x_i)_{i\in n}\in \IR^n$.}$$ 
A Banach space is {\em Hilbert} if its norm is generated by an inner product $\langle\cdot,\cdot\rangle$ in the sence that $\|x\|=\sqrt{\langle x,x\rangle}$ for all $x\in X$. In particular, for every $n\in\IN$ the Banach space $\ell^n_2$ is Hilbert: its norm is generated by the inner product $$\langle x,y\rangle=\sum_{i\in n} x_iy_i.$$
It is well-known that the inner product $\langle\cdot,\cdot\rangle$ of a Hilbert space can be uniquely recovered from the norm by the formula
$$\langle x,y\rangle=\tfrac14(\|x+y\|^2-\|x-y\|^2).$$
For a subset $A\subseteq X$ of  a Hilbert space $X$ we denote by
$$A^\perp=\{x\in X:\forall a\in A\;\;\langle x,a\rangle=0\}$$the {\em orthogonal complement} of $A$ in $X$.  

For an operator $T:X\to Y$ between Banach spaces, $\ker(T)$ stands for the {\em kernel} $T^{-1}(0)$ of $T$.

\begin{lemma}\label{l:Hilbert-op} Let $T:X\to Y$ be an operator from a Hilbert space $X$ to a Banach space $Y$ and $K=\ker(T)$ be the kernel of $T$. Then $T[B_X(1)]=T[B_X(1)\cap K^\perp]$.
\end{lemma}

\begin{proof} The inclusion $T[B_X(1)\cap K^\perp]\subseteq T[B_X(1)]$ is trivial. To prove that $T[B_X(1)]\subseteq T[B_X(1)\cap K^\perp]$, fix any element $y\in T[B_X(1)]$ and find $x\in B_X(1)$ with $T(x)=y$. Let $x'$ be the orthogonal projection of $x$ onto the subspace $K$. Then $x-x'\in K^\perp$ and $T(x-x')=T(x)-T(x')=T(x)-0=y$. Since $x=x'+(x-x')$ and $\langle x',x-x'\rangle=0$ in the Hilbert space $X$, we have {$\|x-x'\|^2=\|x\|^2-\|x'\|^2\le \|x\|^2<1$} by the Pitagoras Theorem. Therefore,  $y=T(x-x')\in T[B_X(1)\cap K^\perp]$.
\end{proof}

\begin{lemma}\label{l:Hilbert-ort} Let $X$ be a Hilbert space of finite dimension $n>0$, and $T:X\to Y$ be an operator to a Banach space $Y$. Then 
$$\|T\|\le\sqrt{n}\max_{1\le i\le n}\|T(e_i)\|$$for any orthonormal basis $e_1,\dots,e_n$ in $X$.
\end{lemma}

\begin{proof}  Consider the isomorphism $I:\ell^n_1\to X$ defined by $I:(x_1,\dots,x_n)\mapsto\sum_{i=1}^nx_ie_i$. It is easy to see that $\|I\|=1$. Using Cauchy--Bunyakovsky--Schwartz inequality, one can show that $\|I^{-1}\|=\sqrt{n}$.
It follows that $$
\|T\|=\|T\circ I\circ I^{-1}\|\le \|T\circ I\|\cdot\|I^{-1}\|=\max_{1\le i\le n}\|T(e_i)\|\cdot\sqrt{n}.
$$ 
\end{proof}

The {\em Banach-Mazur distance} between two isomorphic Banach spaces $X,Y$  is defined as
$$
d_{BM}(X, Y):= \inf\{\|I \| \cdot\|I^{-1}\|: \mbox{$I$ is an isomorphism between $X$ and $Y$}\}.
$$

The following proposition is a classical result of John \cite{john} (see also  \cite[Theorem 2.3.2]{alm}).	

\begin{proposition}[John]\label{p:John} For any $n$-dimensional Banach space $X$, we have $d_{BM}(X,\ell_2^n)\leq \sqrt{n}$.
\end{proposition}


\begin{lemma}\label{l:isoper} Let $\e,\delta$ be positive real numbers. Let $i,j:X\to Y$ be two operators between Banach spaces. If $\|i-j\|\le\delta$ and $i$ is an $\e$-isometric embedding, then $j$ is an $(\e+\delta)$-isometric embedding.
\end{lemma}

\begin{proof} Given any nonzero $x\in X$, observe that
$$\|j(x)\|\le \|i(x)\|+\|i(x)-j(x)\|<(1+\e)\|x\|+\|i-j\|\cdot\|x\|\le (1+\e)\|x\|+\delta\|x\|=(1+\e+\delta)\|x\|$$
and
$$\|j(x)\|\ge \|i(x)\|-\|i(x)-j(x)\|>(1-\e)\|x\|-\delta\|x\|=(1-\e-\delta)\|x\|,$$
which means that $j$ is an $(\e+\delta)$-isometric embedding.
\end{proof}

\begin{lemma}\label{open} If $G:\IG\to\IG$ is a \Gurarii operator, then $G[B_\G(1)]=B_\G(1)$ and hence $\|G\|=1=\|G^\leftarrow\|$.
\end{lemma}

\begin{proof} The nonexpansive property of $G$ ensures that $G[B_\G(1)]\subseteq B_\G(1)$. To show that $B_\G(1) \subseteq G[B_\G(1)]$, fix any element $y\in B_{\G}(1)$. If $\|y\|=0$, then $y=0=G(0)\in G[B_\G(1)]$. So, we assume that $\|y\|>0$.

Let $X=Y=\IR$ and $T:X\to Y$ be the identity operator. Let $X_0=\{0\}\subset X$, $Y_0=Y$, $i_0:X_0\to\{0\}\subseteq\G$ be the trivial isometric embedding and $j:Y_0\to\IG$ be the isometric embedding defined by $j(t)=\frac{ty}{\|y\|}$.
Since $\|y\|<1$, there exists $\e>0$ such that $(1+\e)\|y\|<1$. 
Since the operator $G$ is \Gurarii\!, there exists an $\e$-isometric embedding $i:X\to\IG$ such that $G\circ i=j\circ T$. Then the element $x=i(\|y\|)$ has $G(x)=G\circ i(\|y\|)=j\circ T(\|y\|)=j(\|y\|)=y$ and $\|x\|=\|i(\|y\|)\|< (1+\e)\|y\|<1$. Therefore, $y=G(x)\in G[B_\G(1)]$ and hence $B_\G(1)=G[B_{\G}(1)]$. The latter equality implies that $\|G\|=1=\|G^\leftarrow\|$.
\end{proof}

In the proof of Theorem~\ref{t:main} we shall apply the following homogeneity property of the \Gurarii space, proved by Kubi\'s and Solecki in \cite[1.1]{KS}.

\begin{lemma}\label{l:KS} Let $X$ be a finite-dimensional linear subspace of the \Gurarii space $\G$. For every $\e>0$ and every $\e$-isometric embedding $f:X\to\G$ there exists a bijective isometry $I:\G\to\IG$ such that $\|f-I{\restriction}_X\|<\e$.
\end{lemma}

\section{Characterizing \Gurarii operators}

In this section we shall present several characterizations of \Gurarii operators.
The following characterization was proved in \cite{GK}.

\begin{theorem}\label{t:GK} An operator $G:\G\to\G$ is \Gurarii if and only if it satisfies the following condition:
\begin{itemize}
\item[{(GA)}] for any $\eps>0$, any nonexpansive operator $ {T:X\to Y}$ between finite-dimensional Banach spaces, any Banach subspaces $X_0 \subseteq X$, $Y_0\subseteq Y$ with $T[X_0]\subseteq Y_0$, and any isometric embeddings $ i_0: {X_0}\to \G$, $ j_0: {Y_0}\to \G$ with $G \circ i_0= j_0 \circ T{\restriction}_{X_0}$, there exist $\eps$-isometric embeddings $ {i}: X\to \G$, ${j} :Y\to \G$ such that		$$\|i{\restriction}_{X_0} - i_0\| <\eps, \quad   \|j{\restriction}_{Y_0}- j_0\| < \eps, \quad \operatorname{and}\quad \|G\circ i  - j \circ T\| < \eps.$$
\end{itemize}
\end{theorem}

Next, we show that the isometric embeddings $i_0,j_0$ in the above characterization can be replaced by $\delta$-isometric embeddings for a sufficiently small $\delta$.
For non-negative real numbers $\e,n,m,t$ let $\delta(\e,n,m,t)$ be the largest number $\delta$ such that
$$\delta\big(1+(1+\delta)(n+1)\max\{1,t\sqrt{nm}\}\big)\le\e.$$

\begin{theorem}\label{t:mainB} A nonexpansive operator $G:\G\to\G$ is \Gurarii if and only if it satisfies the condition:
\begin{enumerate}
		\item[{(GB)}] for any $\eps\in(0,1]$, any nonexpansive operator $ {T:X\to Y}$ between finite-dimensional Banach spaces, any Banach subspaces $X_0 \subseteq X$, $Y_0\subseteq Y$ with $T_0[X_0]\subseteq Y_0$ where $T_0=T{\restriction}_{X_0}:X_0\to Y_0$, and any $\delta$-isometric embeddings $ i_0: {X_0}\to \G$, $ j_0: {Y_0}\to \G$ with $$\|G \circ i_0 - j_0 \circ T_0\|<\delta:=\delta(\e,\dim(X_0),\dim(Y_0),\|T_0^\leftarrow\|)$$ there exist $\eps$-isometric embeddings $ {i}: X\to \G$, ${j} :Y\to \G$ such that		$$\|{i{\restriction}_{X_0} - i_0} \| < \eps, \quad   \|j{\restriction}_{Y_0}- j_0\| < \eps, \quad \operatorname{and}\quad \|G\circ i  - j \circ T\| < \eps.$$
	\end{enumerate}
\end{theorem}

\begin{proof} The ``if'' part follows immediately from Theorem~\ref{t:GK} and the trivial implication (GB)$\Ra$(GA). 

To prove the ``only if'' part, we need to show that every \Gurarii operator $G:\IG\to\IG$ has the property (GB). Fix any $\eps\in(0,1]$, any nonexpansive operator $T:X\to Y$ between finite-dimensional Banach spaces, any Banach subspaces $X_0 \subseteq X$ and $Y_0\subseteq Y$ with $T[X_0]\subseteq Y_0$, and any $\delta$-isometric embeddings $ i_0: X_0\to \G$, $j_0: Y_0\to \G$ such that $\|G\circ i_0 - j_0 \circ T_0\|<\delta$ where $\delta= \delta(\e,\dim(X_0),\dim(Y_0),\|T_0^\leftarrow\|)$ and $T_0=T{\restriction}_{X_0}:X_0\to Y_0$.

Let $n=\dim X_0$ and $m=\dim Y_0$. By Proposition~\ref{p:John}, there exist isomorphisms $I:X_0\to \ell_2^n$ and $J:Y_0\to\ell^m_2$ such that $\|I\|\cdot\|I^{-1}\|\le\sqrt{n}$ and $\|J\|\cdot\|J^{-1}\|\le\sqrt{m}$.


Fix a basis $e_1,\cdots, e_{k}, e_{k+1}, \cdots, e_n$ in $X_0$ and a basis $e_1',\dots, e_{n-k}',e_{n-k+1}' \dots, e_m'$ in $Y_0$ such that:
\begin{enumerate}
\item $I(e_1),\dots, I(e_n)$ is an orthonormal basis for the Hilbert space $\ell_2^n$;
\item $e_1,\dots,e_k$ is a basis for the subspace $\ker(T_0)$;
\item $e'_l=T(e_{l+k})$ for $l\in\{1, \dots, n-k\}$;
\item $J(e_{n-k+1}'), \dots, J(e_m')$ is an orthonormal basis for the orthogonal complement of the space $J[T[X_0]]$ in the Hilbert space $\ell_2^m$.
\end{enumerate}

Let $K=I[\ker(T_0)]\subseteq\ell^n_2$ and $K^\perp$ be the orthogonal complement of $K$ in the Hilbert space $\ell^n_2$. The choice of the basis $e_1,\dots,e_n$ guarantees that $I(e_1),\dots,I(e_k)$ and $I(e_{k+1}),\dots,I(e_{n})$ are orthonormal bases for the spaces $K$ and $K^\perp$, respectively. Observe that the operator $T_\perp=T\circ I^{-1}{\restriction}_{K^\perp}:K^\perp\to T[X_0]$ is an isomorphism. By Lemma~\ref{l:Hilbert-op},
$$T_\perp[B_{K^\perp}(1)]=T\circ I^{-1}[B_{K^\perp}(1)]=T\circ I^{-1}[B_{\ell^n_2}(1)].$$
Also
$I[B_{X_0}(1)]\subseteq B_{\ell^n_2}(\|I\|)$ and hence $B_{X_0}(r)\subseteq I^{-1}[B_{\ell^n_2}(\|I\|r)]$ for every $r\in\IR_+$. Then
\begin{multline*}
\{r\in\IR_+:B_{T[X_0]}(1)\subseteq T[B_{X_0}(r)]\}\subseteq\\ \{r\in\IR_+:B_{T[X_0]}(1)\subseteq T\circ I^{-1}[B_{\ell^n_2}(\|I\|r)]\}=\\
\{r\in\IR_+:B_{T[X_0]}(1)\subseteq T_\perp[B_{K^\perp}(\|I\|r)]\}
\end{multline*}
and
thus
$$
\begin{aligned}
\|T_0^{\leftarrow}\|&=\inf\{r\in\IR_+:B_{T[X_0]}(1)\subseteq T[B_{X_0}(r)]\}\ge\\
&\ge\inf\{r\in\IR_+:B_{T[X_0]}(1)\subseteq T_\perp[B_{K^\perp}(\|I\|r)]\}=\\
&=\|I\|^{-1}\inf\{\|I\|r:B_{T[X_0]}(1)\subseteq T_\perp[B_{K^\perp}(\|I\|r)]\}=
\|I\|^{-1}\|T_\perp^{-1}\|.
\end{aligned}
$$
Therefore, $$\|T_\perp^{-1}\|\le\|T_0^{\leftarrow}\|\cdot\|I\|.$$

By condition (2), for any $l\in \{1,\dots, k\}$ we have $e_l \in \ker(T_0)$ and hence
\begin{align*}
\| G(i_0(e_l))\| & = \| G \circ i_0(e_l) -j_0\circ T_0(e_l)\| \leq \| G\circ i_0 -j_0\circ T_0\|\cdot \|e_l\| <\\
&< \delta\cdot \|I^{-1}\| \cdot \|I(e_l)\| =\delta\cdot \|I^{-1}\|.
\end{align*}

Lemma~\ref{open} implies that 
$$G\big[i_0(e_l)+B_\IG(\delta\|I^{-1}\|)\big]=G(i_0(e_l))+B_\IG(\delta\|I^{-1}\|) \ni 0.$$
Then there exist $y_l \in i_0(e_l)+B_\IG(\delta\|I^{-1}\|)$ such that $y_l\in \ker G$.

Consider the operator $i_0':X_0\to \G$ such that $i'_0(e_l)=y_l$ for all $l\in\{1,\dots,k\}$ and  $i'_0(e_l)=i_0(e_l)$ for all $l\in\{k+1,\dots,n\}$. 

\begin{claim}\label{cl:i0'} $\|i'_0-i_0\|<\delta n$ and $i'_0$ is a $\delta(n+1)$-isometric embedding.
\end{claim}

\begin{proof}
By Lemma~\ref{l:Hilbert-ort},
\begin{align*}
\| i'_0-i_0\| & =\| (i_0'-i_0)\circ I^{-1}\circ I\|\leq  \| (i_0'-i_0)\circ I^{-1}\| \cdot \| I\| \leq \\
&\leq \|I\|\sqrt{n} \max_{1\leq l\leq n}  \|(i'_0-i_0)\circ I^{-1}(I(e_l))\| = \|I\|\sqrt{n}\max_{1\le l\le n} \| i'_0(e_l)-i_0(e_l)\|=\\
&= \|I\|\sqrt{n} \max_{1\le l\le k} \| y_l-i_0(e_l)\|< \|I\|\sqrt{n}\delta\|I^{-1}\|\le\delta n.
\end{align*}

Since $i_0$ is a $\delta$-isometric embedding and $\|i_0-i'_0\|<\delta n$, the operator $i'_0:X_0\to\IG$ is a $\delta(n+1)$-isometric embedding by Lemma~\ref{l:isoper}.
\end{proof}

Let $j'_0:Y_0\to \G$ be a unique operator such that  $j'_0(e'_l)=G\circ i'_0(e_{l+k})=G\circ i_0(e_{l+k})$ for all $l\in\{1,\dots,n-k\}$ and $j'_0(e'_l)=j_0(e'_l)$ for all $l\in\{n-k+1,\dots,m\}$.
Observe that for every $l\in \{1,\dots,k\}$ we have $$j'_0\circ T_0(e_l)=j'_0(0)=0=G(y_l)=G\circ i'_0(e_l),$$ and for every $l\in\{k+1,\dots,n\}$ we have 
$$j'_0\circ T_0(e_l)=j'_0(e'_{l-k})=G\circ i'_0(e_{l-k+k})=G\circ i'_0(e_l).$$
Therefore, $j'_0\circ T_0(e_l)=G\circ i'_0(e_l)$ for every $l\in\{1,\dots,n\}$. Taking into account that $e_1,\dots,e_n$ is a basis of the space $X_0$, we conclude that \begin{equation}\label{eq:rivne}
j_0'\circ T_0=G\circ i'_0.
\end{equation}  

\begin{claim}\label{cl:j0'} $\|j'_0-j_0\|\le\delta(n+1)\sqrt{nm}\|T_0^{\leftarrow}\|$ and $j_0'$ is a $\delta(1+(n+1)\sqrt{nm}\|T_0^\leftarrow\|)$-isometric embedding.
\end{claim}

\begin{proof} Since $j'_0\circ T_0=G\circ i'_0$, we have 
\begin{multline*}
\|(j'_0-j_0)\circ T_0\|= \|G\circ i'_0-j_0\circ T_0\|\le\|G\circ i'_0-G\circ i_0\|+\|G\circ i_0-j_0\circ T_0\|\\
<\|G\|\cdot\|i'_0-i_0\|+\delta\le\delta n+\delta=\delta(n+1)
\end{multline*}
and
$$
\begin{aligned}
&\|(j'_0-j_0){\restriction}_{T[X_0]}\|=\|(j'_0-j_0)\circ T_\perp\circ T_\perp^{-1}\|=\\
&=\|(j'_0-j_0)\circ T_0\circ I^{-1}\circ T_\perp^{-1}\|\le
\|(j'_0-j_0)\circ T_0\|\cdot\|I^{-1}\|\cdot\|T_\perp^{-1}\|\le \\
&\le\delta(n+1)\|I^{-1}\|\cdot \|T_0^{\leftarrow}\|\cdot\|I\|\le \delta(n+1)\sqrt{n}\|T_0^{\leftarrow}\|.
\end{aligned}
$$

Given any element $z\in Y_0$, write $J(z)$ as $J(z)=J(x)+J(y)$ for some $x\in T[X_0]$ and $y\in Y_0$ such that $J(y)\in(J[T[X_0]])^\perp$. Since the vectors $J(x)$ and $J(y)$ are orthogonal in the Hilbert space $\ell^m_2$, the Pitagoras Theorem ensures that $\|J(x)\|\le\|J(z)\|$. Taking into account that $y$ belongs to the linear hull of the vectors $e'_{n-k+1},\dots,e'_m$ and $j_0'(e_l')=j_0(e_l')$ for all $l\in\{n-k+1,\dots,m\}$, we conclude that $j'_0(y)=j_0(y)$.
Then
$$
\begin{aligned}
\|j'_0(z)-j_0(z)\|&=\|j'_0(x+y)-j_0(x+y)\|=
\|j'_0(x)-j_0(x)\|=\|(j'_0-j_0)\circ J^{-1}\circ J(x)\|\le\\
&\le\|(j'_0-j_0){\restriction}_{T[X_0]}\|\cdot\|J^{-1}\|\cdot\|J(x)\|\le\|(j'_0-j_0){\restriction}_{T[X_0]}\|\cdot\|J^{-1}\|\cdot\|J(z)\|
\le\\
&\le\delta(n+1)\sqrt{n}\|T_0^{\leftarrow}\|\cdot\|J^{-1}\|\cdot\|J\|\cdot\|z\|\le \delta(n+1)\sqrt{n}\|T_0^{\leftarrow}\|\sqrt{m}\cdot\|z\|.
\end{aligned}
$$
and hence
$$\|j'_0-j_0\|\le\delta(n+1)\sqrt{nm}\,\|T_0^{\leftarrow}\|.$$
Since $j_0$ is a $\delta$-isometric embedding, the operator $j'_0$ is a $\delta(1+(n+1)\sqrt{nm}\|T_0^\leftarrow\|)$-isometric embedding by Lemma~\ref{l:isoper}.
\end{proof}

Let $$r=\frac1{1+\delta(n+1)\max\{1,\sqrt{nm}\|T_0^{\leftarrow}\|\}},$$ 
$D_X$ be the convex hull of the set $(i'_0)^{-1}[B_\G(1)]\cup B_X(r)$ in $X$ and $D_Y$ be the convex hull of the set $(j'_0)^{-1}[B_\G(1)]\cup B_Y(r)$ in $Y$.

\begin{claim}\label{cl:BXD} $B_X\big(r)\subseteq D_X\subseteq B_X\big(\frac1{1-\delta(n+1)}\big)$ and $D_X\cap X_0=(i'_0)^{-1}[B_\G(1)]$. 
\end{claim}

\begin{proof} The inclusion $B_X(r)\subseteq D_X$ follows from the definition of $D_X$. To see that $D_X\subseteq B_X(\frac1{1-\delta(n+1)})$, take any $x\in (i'_0)^{-1}[B_\G(1)]$. By Claim~\ref{cl:i0'}, $i_0'$ is a $\delta(n+1)$-isometric embedding and hence $(1-\delta(n+1))\|x\|\le \|i'_0(x)\|<1$ and $\|x\|<\frac1{1-\delta(n+1)}$. Then 
$$D_X\subseteq \conv\big((i'_0)^{-1}[B_\G(1)]\cup B_X(r)\big)\subseteq B_X\big(\tfrac1{1-\delta(n+1)}\big).$$

By Claim~\ref{cl:i0'}, $i'_0[B_{X_0}(r)]\subseteq B_\G(r\|i'_0\|)\subseteq B_\G\big(r(1+\delta(n+1))\big)\subseteq B_\G(1)$ and 
$i'_0[(i'_0)^{-1}[B_\G(1)]]\subseteq B_\G(1)$, which implies that $i'_0[D_X\cap X_0]\subseteq B_\G(1)$ and $D_X\cap X_0=(i'_0)^{-1}[B_\G(1)]$ by the injectivity of $i'_0$. 
\end{proof}

By analogy, we can apply Claim~\ref{cl:j0'} and prove

\begin{claim}\label{cl:BYD} $B_Y(r)\subseteq D_Y\subseteq B_Y\big(\tfrac1{1-\delta(1+(n+1)\sqrt{nm}\|T_{0}^\leftarrow\|)}\big)$  and $D_Y\cap Y_0=(j'_0)^{-1}[B_\G(1)].$
\end{claim}

\begin{claim}\label{cl:TD} $T[D_X]\subseteq D_Y$.
\end{claim}

\begin{proof} Observe that for every $x\in (i'_0)^{-1}[B_\G(1)]$, the equality $j'_0\circ T_0=G\circ i'_0$ proved in (\ref{eq:rivne}) and the inequality $\|G\|\le 1$ imply $$j'_0\circ T_0(x)=G\circ i_0'(x)\in G[B_\G(1)]\subseteq B_\G(1)$$ and hence $T(x)=T_0(x)\in (j'_0)^{-1}[B_\G(1)]\subseteq D_Y$. On the other hand, for every $x\in B_X(r)$ the inequality $\|T\|\le1$ implies $T(x)\in T[B_X(r)]\subseteq B_Y(r)\subseteq D_Y$. Then 
\begin{multline*}
T[D_X]=T\big[\conv\big((i'_0)^{-1}[B_\G(1)]\cup B_X(r)\big)\big]=\\
\conv\big(T[(i'_0)^{-1}[B_\G(1)]\cup B_X(r)]\big)\subseteq \conv(D_Y)=D_Y.
\end{multline*}
\end{proof}

Let $X'$ be the space $X$ endowed with the norm $\|\cdot\|_{X'}$ defined by
$$\|x\|_{X'}=\inf\{t\in\IR_+:x\in t\cdot D_X\},$$
and let $Y'$ be the space $Y$ endowed with the norm $\|\cdot\|_{Y'}$
$$\|y\|_{Y'}=\inf\{t\in \IR_+:y\in t\cdot D_Y\}.$$

It follows that $B_{X'}(1)=D_X$ and $B_{Y'}(1)=D_Y$.
Let $\Id_{X,X'}:X\to X'$, $\Id_{X',X}:X'\to X$, $\Id_{Y,Y'}:Y\to Y'$, and $\Id_{Y',Y}:Y'\to Y$ be the identity operators.




In the Banach spaces $X'$ and $Y'$ consider the subspaces $X'_0=\Id_{X,X'}[X_0]$ and $Y_0'=\Id_{Y,Y'}[Y_0]$.  
Let $\Id_{X_0,X'_0}:X_0\to X_0'$, $\Id_{X_0',X_0}:X_0'\to X_0$, $\Id_{Y_0,Y_0'}:Y\to Y_0'$, and $\Id_{Y_0',Y_0}:Y_0'\to Y_0$ be the identity operators.

Claims~\ref{cl:BXD} and \ref{cl:BYD} imply that $i'_0\circ \Id_{X_0',X_0}:X'_0\to \G$ and $j'_0\circ \Id_{Y_0',Y_0}:Y_0'\to\G$ are isometric embeddings, and Claim~\ref{cl:TD} implies that the operator $T'=\Id_{Y,Y'}\circ T\circ\Id_{X',X}:X'\to Y'
  $ is nonexpansive. Let $T'_0=T'{\restriction}_{X'_0}$. Since $G\circ i'_0\circ \Id_{X_0',X_0}=j'_0\circ T'_0$, we can apply the definition of the \Gurarii space and find $\delta$-isometric embeddings $i':X'\to\G$ and $j':Y'\to\IG$ such that $\|i'{\restriction}_{X'_0}-i'_0\circ \Id_{X'_0,X_0}\|<\delta$, $\|j'{\restriction}_{Y'_0}-j'_0\circ\Id_{Y'_0,Y_0}\|<\delta$ and $\|G\circ i'-j\circ T'\|<\delta$. Consider the operators $i=i'\circ\Id_{X,X'}:X\to\G$ and $j=j'\circ\Id_{Y,Y'}:Y\to\G$. It remains to prove that $i,j$ are $\e$-isometric embeddings and 
  $$\max\{\|i{\restriction}_{X_0}-i_0\|,\|j{\restriction}_{Y_0}-j_0\|,\|G\circ i-j\circ T\|\}<\e.$$
To see that $i$ is an $\e$-isometric embedding, take any nonzero $x\in X$ and applying Claim~\ref{cl:i0'}, conclude that
\begin{multline*}
\|i(x)\|_{\G}=\|i'\circ \Id_{X,X'}(x)\|_{\G}\le\|i'\|\cdot\|\Id_{X,X'}(x)\|_{X'}<(1+\delta)\tfrac1r\|x\|_{X}=\\
(1+\delta)\cdot(1+\delta(n+1)\max\{1,\sqrt{nm}\|T_0^\leftarrow\|\})\|x\|_X\le(1+\e)\|x\|_X,
 \end{multline*}
 by the choice of $\delta=\delta(\e,n,m,\|T_0^\leftarrow\|)$.
On the other hand the inclusion $D_X\subseteq B_X(\frac1{1-\delta(n+1)})$ proved in Claim~\ref{cl:BXD} implies
 $$\|i(x)\|_{\G}=\|i'\circ\Id_{X,X'}(x)\|_\G>(1-\delta)\|\Id_{X,X'}(x)\|_{X'} \ge (1-\delta)(1-\delta(n+1))\|x\|_X\ge (1-\e)\|x\|_X$$ by the choice of $\delta=\delta(\e,n,m,\|T_0^\leftarrow\|)$. This means that $i:X\to\G$ is an $\e$-isometric embedding.
 
By analogy we can prove that $j$ is an $\e$-isometric embedding.

The inclusion $B_X(r)\subseteq D_X$ implies $\|\Id_{X_0,X_0'}\|\le\|\Id_{X,X'}\|\le\frac1r$ and hence
$$\|i{\restriction}_{X_0}-i_0\|=\|i'\circ \Id_{X_0,X_0'}-i'_0\circ\Id_{X_0,X'_0}\|\le \|i'-i'_0\|\cdot\|\Id_{X_0,X'_0}\|< \frac\delta{r}\le\e.$$
By analogy we can prove that $\|j{\restriction}_{Y_0}-{j}_0\|<\e$.
Finally, observe that 
\begin{multline*}
\|G\circ i-j\circ T\|=\|G\circ i'\circ \Id_{X,X'}-j'\circ  \Id_{Y,Y'}\circ T\|=\|G\circ i'\circ \Id_{X,X'}-j'\circ   T'\circ\Id_{X,X'}\|\le\\
\le\|G\circ i'-j'\circ T'\|\cdot\|\Id_{X,X'}\|< \delta\cdot \|\Id_{X,X'}\|\le\frac\delta{r}\le\e.
\end{multline*}
\end{proof}

Next, we show that the spaces $X,Y$, operator $T$, subspaces $X_0,Y_0$ and $\delta$-isometries $i_0,j_0$ in Theorem~\ref{t:mainB} can be chosen from a suitable countable family of standard objects, defined below.

A {\em standard linear space} is the linear space $\IR^n$ for some $n\in\w$. For a standard linear space $X=\IR^n$ by $E_X=\{x\in\{0,1\}^n:|x^{-1}(1)|=1\}$ we denote the standard unit basis in $X=\IR^n$. A linear subspace $X_0$ of a standard linear space $X$ is called {\em standard} if $X_0$ coincides with the linear hull of some subset $E\subseteq E_X$. It is clear that each standard linear space $X$ has exactly $2^{\dim(X)}$ standard linear subspaces.

An operator $T:X\to Y$ between linear spaces is called {\em standard} if there exist numbers $k,n,m$ such that $X=\IR^n$, $Y=\IR^m$, $0\le n-k\le m$ and $T$ assigns to each $x\in\IR^n$ the vector $y=T(x)$ defined by
$$y(i)=\begin{cases}
x(i+k)&\mbox{if $i<n-k$}\\
0&\mbox{otherwise}.
\end{cases}
$$
In this case $\ker(T)=\IR^{k}\times\{0\}^{n-k}$ and $T[X]=\IR^{n-k}\times\{0\}^{m-n+k}$. 

A norm $\|\cdot\|$ on a standard linear space $X=\IR^n$ is called {\em rational} if its closed unit ball $\{x\in X:\|x\|\le1\}$ coincides with the convex hull of some finite set $F\subseteq\IQ^n\subseteq\IR^n$. It is clear that the family of rational norms on $\IR^n$ is countable.

A {\em rational Banach space} if a standard linear space endowed with a rational norm.

For a countable dense set $D$ in the \Gurarii space $\G$, let $\F_D$ be the family of all $9$-tuples $(\e,X,Y,T,X_0,Y_0,T_0,i_0,j_0)$ consisting of
\begin{itemize}
\item a rational number $\e\in(0,1]$;
\item rational Banach spaces  $X,Y$;
\item a nonexpansive standard operator $T:X\to Y$;
\item standard linear subspaces $X_0\subseteq X$ and $Y_0\subseteq Y$ such that $T[X_0]\subseteq Y_0$;
\item the operator $T_0=T{\restriction}_{X_0}:X_0\to Y_0$;
\item $\delta(\e,\dim(X_0),\dim(Y_0),\|T_0^\leftarrow\|)$-isometric embeddings $i_0:X_0\to\IG$ and $j_0:Y_0\to\IG$ such that $i_0(E_X\cap X_0)\cup j_0(E_{Y}\cap Y_0)\subseteq D$.
\end{itemize}
It is easy to see that the family $\F_D$ is countable.

\begin{theorem}\label{t:GC} Let $D$ be a countable dense subset of the \Gurarii space $\G$.  A nonexpansive operator $G:\G\to\G$ is \Gurarii if and only if it satisfies the condition:
\begin{enumerate}
		\item[{(GC)}] for any $9$-tuple $(\e,X,Y,T,X_0,Y_0,T_0,i_0,j_0)\in\F_D$ with $$\|G\circ i_0-j_0\circ T{\restriction}_{X_0}\|<\delta(\e,\dim(X_0),\dim(Y_0),\|T_0^\leftarrow\|),$$ there exist $\eps$-isometric embeddings $ {i}: X\to \G$, ${j} :Y\to \G$ such that		$$\|i{\restriction}_{X_0} - i_0 \| < \eps, \quad   \|j{\restriction}_{Y_0}- j_0\| < \eps, \quad \operatorname{and}\quad \|G\circ i  - j \circ T\| < \eps.$$
	\end{enumerate}
\end{theorem}

\begin{proof} The ``only if'' part follows immediately from the ``only if'' part of Theorem~\ref{t:mainB}. To prove the ``if'' part, assume that the condition (GC) is satisfied. By Theorem~\ref{t:mainB}, the \Gurarii property of the operator $G$ will follow as soon as we show that $G$ satisfies the condition (GB). To this end, fix any $\e\in(0,1]$, any nonexpansive operator $T:X\to Y$ between finite-dimensional Banach spaces, any linear subspaces $X_0\subseteq X$ and $Y_0\subseteq Y$ with $T[X_0]\subseteq Y_0$, and $\delta$-isometric embeddings $i_0:X_0\to \G$ and $j_0:Y_0\to\IG$ such that $\|G\circ i_0-j_0\circ T_0\|<\delta:=\delta(\e,\dim(X_0),\dim(Y_0),\|T_0^\leftarrow\|)$, where $T_0=T{\restriction}_{X_0}:X_0\to Y_0$. 

Choose a basis $E_X=\{e_1,\dots,e_{\dim(X)}\}$ in the space $X$ such that 
\begin{itemize}
\item $\{e_1,\dots,e_{\dim(\ker(T_0))}\}$ is a basis for the linear space $\ker(T_0)$;
\item $\{e_1,\dots,e_{\dim(\ker(T))}\}$ is a basis for the linear space $\ker(T)$;
\item $\{e_1,\dots,e_{\dim(\ker(T_0))}\}\cup\{e_{\dim(\ker(T))+1},\dots,e_{\dim(\ker(T))+\dim(X_0)-\dim(\ker(T_0))}\}$ is a basis for the space $X_0$;
\item $\{e_1,\dots,e_{\dim(T^{-1}[Y_0])}\}$ is a basis for the linear space $T^{-1}[Y_0]$.
\end{itemize} 
Let $E_Y=\{e'_1,\dots,e'_{\dim(Y)}\}$ be a basis for the linear space $Y$ such that $e'_i=T(e_{\dim(\ker(T))+i})$ for all $i\in\{1,\dots,\dim(X)-\dim(\ker(T))\}$. The bases $E_X$ and $E_Y$ allow us to identify the spaces $X,Y$ with the standard linear spaces $\IR^{\dim(X)}$ and $\IR^{\dim(Y)}$. The choice of the bases $E_X,E_Y$ ensures that the operator $T:X\to Y$ is standard and the subspaces $X_0\subseteq X$ and $Y_0\subseteq Y$ are standard. 

By the compactness of the unit spheres in the Banach spaces $X_0$, $Y_0$, there exists $\delta'<\delta$ such that the $\delta$-isometric embeddings $i_0:X_0\to \IG$ and $j_0:Y_0\to\G$ are $\delta'$-isometric embeddings. Since $\|G\circ i_0-j_0\circ T_0\|<\delta$, we can additionally assume that $\|G\circ i_0-j_0\circ T_0\|<\delta'$.

By the continuity of the function $\delta(\cdot,\dim(X_0),\dim(Y_0),\cdot)$, there exists a positive real number $\nu$ such that $$\delta'+2\nu\le\delta(\e',\dim(X_0),\dim(Y_0),t)$$ for any $\e',t\in\IR$ with $$\frac{\e-2\nu}{1+\nu}\le\e'\le\e\quad\mbox{and}\quad\|T_0^\leftarrow\|\le t\le (1+\nu)^2\|T_0^\leftarrow\|.$$ By the density of $D$ in $\G$, the $\delta'$-isometric embeddings $i_0:X_0\to\G$ and $j_0:Y_0\to\G$ can be approximated by $\delta'$-isometric embeddings $i''_0:X_0\to\G$ and $j''_0:Y_0\to\G$ such that $i''_0(E_X\cap X_0)\cup j''_0(E_Y\cap Y_0)\subseteq D$, $\max\{\|i''_0-i_0\|,\|j''_0-j_0\|\}<\nu$, and $\|G\circ i''_0-j_0''\circ T_{0}\|<\delta'$. The norms $\|\cdot\|_X$ and $\|\cdot\|_Y$ on the standard linear spaces $X,Y$ can be approximated by rational norms $\|\cdot\|_{X'}:X\to[0,\infty)$ and $\|\cdot\|_{Y'}:Y\to[0,\infty)$ such that for every nonzero elements $x\in X$ and $y\in Y$ the following conditions are satisfied:
\begin{itemize}
\item[(a)] $\|x\|_X\le\|x\|_{X'}<(1+\nu)\|x\|_X$ and $\frac1{1+\nu}\|y\|_Y\le \|y\|_{Y'}\le\|y\|_Y$;
\item[(b)] $(1-\delta')\|x\|_{X'}<\|i_0'(x)\|_\G<(1+\delta')\|x\|_{X'}$ and $(1-\delta')\|y\|_{Y'}<\|j_0'(y)\|_\G<(1+\delta')\|y\|_{X'}$.
\end{itemize}
Let $X'$ (resp. $Y'$) be the rational Banach space $X$ (resp. $Y$) endowed with the rational norm $\|\cdot\|_{X'}$ (resp. $\|\cdot\|_{Y'}$). 
The condition (a) implies
\begin{equation}\label{eq:ballsXY}
B_{X'}(1)\subseteq B_X(1)\subseteq B_{X'}(1+\nu)\mbox{ \ and \ }B_{Y}(1)\subseteq B_{Y'}(1)\subseteq B_Y(1+\nu).
\end{equation}

Let $$\Id_{X,X'}:X\to X',\quad \Id_{X',X}:X'\to X,\quad \Id_{Y,Y'}:Y\to Y',\quad\mbox{and}\quad \Id_{Y',Y}:Y'\to Y$$ be the identity operators. Let $X'_0=\Id_{X,X'}[X_0]$, $Y'_0=\Id_{Y,Y'}[Y_0]$, and 
$$\Id_{X_0,X_0'}:X_0\to X_0',\quad \Id_{X'_0,X_0}:X'_0\to X_0,\quad\Id_{Y_0,Y'_0}:Y_0\to Y'_0\quad\mbox{and}\quad\Id_{Y'_0,Y_0}:Y'_0\to Y_0$$ be the identity operators. Let $i_0'=i_0''\circ\Id_{X_0',X_0}$ and $j_0'=j_0''\circ\Id_{Y_0',Y_0}$.

The condition (a) implies that $\max\{\|\Id_{X',X}\|,\|\Id_{Y,Y'}\|\}\le 1$ and hence the operator $T'=\Id_{Y,Y'}\circ T\circ\Id_{X',X}:X'\to Y'$ is nonexpansive. Let $T'_0=T'{\restriction}_{X'_0}$. The inclusions (\ref{eq:ballsXY}) imply
$$\|T_0^{\leftarrow}\|\le\|(T_0')^\leftarrow\|\le (1+\nu)^2\|T_0^\leftarrow\|.$$ 

Choose any rational $\e'$ such that $\frac{\e-2\nu}{1+\nu}\le\e'\le\frac{\e-\nu}{1+\nu}$. The choice of $\nu$ ensures that $$\delta'+2\nu\le\delta(\e',\dim(X_0),\dim(Y_0),\|(T_0')^\leftarrow\|).$$ Then the $9$-tuple $(\e',X',Y',T',X_0',Y'_0,T'_0,i_0',j'_0)$ belongs to the family $\F_D$. 
Since 
$$
\begin{aligned}
&\|G\circ i'_0-j'_0\circ T_0'\|\le \|G\circ i''_0-j''_0\circ T_0\|\cdot\|\Id_{X'_0,X_0}\|\le \|G\circ i''_0-j''_0\circ T_0\|\le\\
&\le \|G\circ i''_0-G\circ i_0\|+\|G\circ i_0-j_0\circ T_0\|+\|j_0\circ T_0-j''_0\circ T_0\|\le\\
&\le\|G\|\cdot\|i''_0-i_0\|+\delta'+\|j_0-j_0''\|\cdot\|T_0\|\le \nu+\delta'+\nu\le\delta(\e',\dim(X_0),\dim(Y_0),\|(T_0')^\leftarrow\|).
\end{aligned}
$$

By the condition (GC), there exist $\e'$-isometric embeddings $i':X'\to\G$ and $j':Y'\to G$ such that $$\max\{\|i'{\restriction}_{X_0'}-i'_0\|,\|j'{\restriction}_{Y'_0}-j_0'\\|,\|G\circ i'-j'\circ T'\|\}<\e'.$$
The inequality $\e'(1+\nu)\le\e$ and the condition (a) imply that the maps $i=i'\circ\Id_{X,X'}$ and $j=j'\circ\Id_{Y,Y'}$ are $\e$-isometric embeddings.
It remains to show that
$$\max\{\|i{\restriction}_{X_0}-i_0\|,\|j{\restriction}_{Y_0}-j_0\|,\|G\circ i-j\circ T\|\}<\e.$$

Observe that 
\begin{multline*}
\|i{\restriction}_{X_0}-i_0\|=\|i'\circ\Id_{X_0,X_0'}-i''_0\|+\|i''_0-i_0\|\le\\
\le\|i'{\restriction}_{X_0'}-i''_0\circ\Id_{X_0',X_0}\|\cdot\|\Id_{X_0,X_0'}\|+\|i''_0-i_0\|<\e'(1+\nu)+\nu\le\e.
\end{multline*}
By analogy we can prove that $\|j{\restriction}_{Y_0}-j_0\|<\e$.
Finally,
$$\|G\circ i-j\circ T\|=\|G\circ i'\circ\Id_{X,X'}-j'\circ T'\circ \Id_{X,X'}\|\le\|G\circ i'-j'\circ T'\|\cdot\|\Id_{X,X'}\|\le\e'(1+\nu)<\e.$$
\end{proof}

\section{Proof of Theorem~\ref{t:main}}

First we show that the set $\mathcal G(\G)$ of \Gurarii operators is dense in the space $\mathcal B(\G)$.

 Fix any nonexpansive operator $T\in \mathcal B(\G)$ and any neighborhood $O_T$ of $T$ in $\mathcal B(\G)$. By the definition of the strong operator topology, there exist $\e>0$ and a finite-dimensional subspace $E\subseteq \G$ such that $\{S\in\mathcal B(\G):\|(S-T){\restriction}_E\|\le 2\e\}\subseteq O_T$. Consider the finite-dimensional Banach space $F=T[E]\subseteq \G$. Let $\Omega:\IG\to\IG$ be the \Gurarii operator, constructed in \cite{GK}. By \cite[3.3]{GK}, the operator $\Omega$ is universal. So, there exist isometric embeddings $i:E\to \G$ and $j:F\to \G$ such that $\Omega\circ i=j\circ T{\restriction}_E$.By Lemma~\ref{l:KS}, for the isometric embeddings $i:E\to \G$ and $j:F\to\G$ there exist bijective isometries $I,J\in\Iso(\G)$ such that $\|i-I{\restriction}_E\|<\e$ and $\|j-J{\restriction}_F\|<\e$.  
  Then $S=J^{-1}\circ \Omega\circ I$ is a \Gurarii operator such that for every $x\in E$ we have
$$
\begin{aligned}
&\|S(x)-T(x)\|=\|J^{-1}\circ \Omega\circ I(x)-T(x)\|\le\\
&\le \|J^{-1}\circ\Omega\circ I(x)-J^{-1}\circ\Omega\circ i(x)\|+\|J^{-1}\circ \Omega\circ i(x)-T(x)\|\le\\
&\le \|J^{-1}\|\cdot\|\Omega\|\cdot\|I(x)-i(x)\|+\|J^{-1}\circ j\circ T(x)-T(x)\|=\\
&=\|I(x)-i(x)\|+\|j\circ T(x)-J\circ T(x)\|\le \|I{\restriction}_E-i\|\cdot\|x\|+\|j-J{\restriction}_F\|\cdot\|T\|\cdot\|x\|\le \\
&\le\e\|x\|+\e\|x\|=2\e\|x\|
\end{aligned}
$$and hence $\|(S-T){\restriction}_E\|\le 2\e$. The choice of $\e$ ensures that $S\in O_T\cap \mathcal G(\G)$, witnessing that the subspace $\mathcal G(\G)$ is dense in $\mathcal B(\G)$.

Now we shall prove that the dense set $\mathcal G(\G)$ is of type $G_\delta$ in $\mathcal B(\G)$. Fix any countable dense set $D$ in the \Gurarii space $\G$. For every $9$-tuple $$t=(\e,X,Y,T,X_0,Y_0,T_0,i_0,j_0)\in\F_D$$ and the number $\delta_t=\delta(\e,\dim(X_0),\dim(Y_0),\|T_0^\leftarrow\|)$, consider the set
$$\F_t=\{G\in\mathcal B(\G):\|G\circ i_0-j_0\circ T_0\|\ge\delta_t\}$$ and the set $\U_t$ of all operators $G\in\mathcal B(\G)$ such that $\|G\circ i_0-j_0\circ T_{0}\|<\delta_t$ and there exist $\e$-isometric embeddings $i:X\to\G$ and $j:Y\to\G$ such that $$\max\{\|i{\restriction}_{X_0}-i_0\|,\|j{\restriction}_{Y_0}-j_0\|,\|G\circ i-j\circ T\|\}<\e.$$
It is easy to see that the set $\F_t$ is closed in $\mathcal B(\G)$ and the set 
$\U_t$ is open in $\mathcal B(\G)$. Then their union $(\F_t\cup\U_t)$ is a $G_\delta$-set in $\mathcal B(\G)$. Theorem~\ref{t:GC} ensures that $\mathcal G(\G)=\bigcap_{t\in\F_D}(\F_t\cup\U_t)$ and hence $\mathcal G(\G)$ is a dense $G_\delta$-set in $\mathcal B(\G)$.

\end{document}